\documentclass[reqno,a4paper,12pt]{amsart}

\numberwithin{equation}{section}
\usepackage{amsmath, amsthm, amssymb, amscd, accents, amsfonts}
\usepackage{mathtools}
\usepackage{url}
\usepackage{mathrsfs,dsfont}
\usepackage{datetime}
\usepackage{hyperref}
\usepackage{xcolor}

\usepackage[sort,nocompress]{cite}
\mathtoolsset{showonlyrefs}

\newtheorem{definition}{Definition}[section]
\newtheorem{theorem}[definition]{Theorem}
\newtheorem*{theorem*}{Theorem}
\newtheorem{lemma}[definition]{Lemma}
\newtheorem{corollary}[definition]{Corollary}
\newtheorem{proposition}[definition]{Proposition}

\def\N{{\mathbb N}}

\def\R{{\mathbb R}}

\def\C{{\mathbb C}}
\def\D{{\mathbb D}}

\allowdisplaybreaks

\DeclareMathOperator{\tr}{tr}
\DeclareMathOperator{\sym}{Sym}
\DeclareMathOperator{\herm}{Herm}
\DeclareMathOperator{\lspan}{span}

\DeclareMathOperator{\conv}{conv}

\usepackage[no-math]{fontspec}
\usepackage[cachedir=minted-cache]{minted}

\newmintinline[lean]{lean4}{bgcolor=white}
\newminted[leancode]{lean4}{escapeinside=!!,
                            breaklines,
                            fontsize=\normalsize}
\begin{document}

\title[Grothendieck's theorem for Bessel sequences]{Grothendieck's theorem for Bessel sequences}

\author[Lukas Liehr]{Lukas Liehr}
\address{Department of Mathematics, Bar-Ilan University, Ramat-Gan 5290002, Israel}
\email{lukas.liehr@biu.ac.il}

\author[Mitchell A. Taylor]{Mitchell A. Taylor}
\address{ Department of Mathematics, ETH Z\"urich, R\"amistrasse 101, 8092 Z\"urich, Switzerland}
\email{mitchell.taylor@math.ethz.ch}

\author[Peiyang Yu]{Peiyang Yu}
\address{Department of Mathematics, ETH Z\"urich, R\"amistrasse 101, 8092 Z\"urich, Switzerland}
\email{peiyang.yu@sam.math.ethz.ch}

\date{\today}

\begin{abstract}
We establish a sharp version of Grothendieck's theorem for Bessel sequences. Precisely, given a Bessel sequence $\{ x_j \}_{j\in\mathbb{N}}$ with Bessel bound $1$ in a Hilbert space, we show that there exists functions $\{ f_j \}_{j\in\mathbb{N}}$ belonging to the unit ball of $L^\infty([0,1])$
such that for all $j,k \in \mathbb{N}$ one has
$$
\langle x_j,x_k\rangle
=
\int_0^1 f_j(x)\overline{f_k(x)}\,dx.
$$
As an application, we give an affirmative answer to an extension problem of Olevskii: if $E \subset [0,1]$ is a Lebesgue measurable set such that $[0,1]\setminus E$ has positive measure, then every Bessel sequence in $L^2(E)$ with Bessel bound $1$ extends to an orthonormal system in $L^2([0,1])$ that is bounded by the (optimal) constant $\lambda([0,1]\setminus E)^{-1/2}$ on $[0,1]\setminus E$. A formalization of our main result in Lean 4 accompanies the paper.
\end{abstract}

\maketitle

\section{Introduction and results}

\subsection{}
A classical probabilistic formulation of Grothendieck's theorem \cite{grothendieck1953resume} states that the inner product of a Hilbert space can be represented as correlations of uniformly bounded random variables. Precisely, given a complex Hilbert space $H$, there exist a probability space $(\Omega,\mathbb P)$, maps $\Phi,\Psi :H \to L^\infty(\Omega,\mathbb P)$ that satisfy $\| \Phi(x) \|_\infty \leq \| x \|$ and $\| \Psi(x) \|_\infty \leq \| x \|$ for all $x \in H$, and a universal constant $K$, such that for all $x,y \in H$ one has
\begin{equation}\label{eq:G}
    \langle x,y \rangle = K \int_\Omega \Phi(x) \overline{\Psi(y)} \, d \mathbb P.
\end{equation}
See \cite[Theorem 3.4]{pisier2012grothendieck}. The maps $\Phi$ and $\Psi$ are in general nonlinear. The optimal constant $K$ so that \eqref{eq:G} holds for any Hilbert space is known as Grothendieck's constant. Notice that in the present paper we consider complex Hilbert spaces and therefore $K=K_\C$ is Grothendieck's constant for complex Hilbert spaces (the corresponding constant for real Hilbert spaces is different and satisfies $K_\C < K_\R$).

\subsection{}
A natural question is if in Grothendieck's theorem it is possible to choose $\Phi = \Psi$, while keeping the constant $K$ universal. A necessary condition for the existence of such a $\Phi$ is that for every $N \in \N$ and all $x_1, \dots, x_N \in H$ with $\|x_j\| \leq 1$ there exists a probability space $(\Omega,\mathbb P)$ and $f_1, \dots, f_N \in L^\infty(\Omega,\mathbb P)$ with $\|f_j\|_{L^\infty(\Omega,\mathbb P)}\leq 1$ for all $j$ such that
\begin{equation}\label{eq:GG}
    \langle x_j,x_k\rangle = K(N) \int_\Omega f_j \overline{f_k} \, d\mathbb P, \quad 1\le j,k\le N.
\end{equation}
Here, $K(N)$ denotes the smallest constant for which this conclusion holds for every Hilbert space $H$ and every $x_1,\dots,x_N$ in its unit ball. In contrast to Grothendieck's theorem, achieving \eqref{eq:GG} so that $K(N)$ stays uniformly bounded in $N$ is not possible. Indeed, Kashin and Szarek showed that $K(N)\asymp \log N$ \cite{kashinSzarek2003knaster,kashinSzarek2003gram}; see also \cite[Remark~3.6]{pisier2012grothendieck} for a self-contained proof of this fact.

A corresponding off-diagonal version of this problem, in which \eqref{eq:GG} is only required to hold for $j \neq k$, was studied by Kashin and Szarek \cite{kashinSzarek2003knaster,kashinSzarek2003gram} and by Alon, Makarychev, Makarychev and Naor \cite{alonEtAl2006quadratic}. To formulate their results, let $K_{\neq}(N)$ denote the smallest constant such that for every real Hilbert space $H$ and every $x_1, \dots, x_N \in H$ of norm one there exists a probability space $(\Omega, \mathbb P)$ and $f_1, \dots, f_N \in L^\infty(\Omega,\mathbb P)$ with $\|f_j\|_{L^\infty(\Omega,\mathbb P)}\leq 1$ such that
$$
\langle x_j,x_k\rangle = K_{\neq}(N) \int_\Omega f_j f_k \, d\mathbb P, \quad j \neq k.
$$
Kashin and Szarek proved that $K_{\neq}(N) \gtrsim (\log N)^{1/2}$ \cite{kashinSzarek2003knaster,kashinSzarek2003gram}. The exact order of growth $K_{\neq}(N) \asymp \log N$ was subsequently established by Alon, Makarychev, Makarychev and Naor \cite{alonEtAl2006quadratic}. Hence, the optimal constant for the off-diagonal problem is of order $\log N$ and does not stay bounded either.

We notice that the upper bound $K_{\neq}(N) \lesssim \log N$ was also established in the linear programming literature by Nemirovski, Roos and Terlaky \cite{nemirovskiRoosTerlaky1999} and Megretski \cite{megretski2001}. Charikar and Wirth later developed an algorithmic formulation of the quadratic Grothendieck inequality \cite{charikarWirth2004}. 

\subsection{}\label{Olevskii}
A problem attributed to Olevskii asks if a representation of the inner product with a single map $\Phi$ is possible with $K(N)$ uniformly bounded in $N$ if the vectors $x_1 , \dots, x_N$ have the additional structure of being a Bessel sequence with Bessel bound $1$, i.e.,
$$
\left\|
\sum_{j=1}^N c_jx_j
\right\|^2
\leq
\sum_{j=1}^N|c_j|^2, \quad c_1, \dots, c_N \in \C.
$$
A countable sequence $\{ x_j \}_{j \in \N}$ is called a Bessel sequence with Bessel bound $1$ if the previous inequality holds for every $N\in\N$.
Olevskii's problem is, for instance, stated as an open problem in Pisier's survey on Grothendieck's theorem \cite[Remark 3.7, p.~257]{pisier2012grothendieck}, where it is described as a well-known open question in the theory of bounded orthogonal systems.
The Bessel assumption is well-motivated by a specific extension problem raised in \cite[p.~58]{Olevskii} which we will discuss below and which constitutes the original motivation and formulation of Olevskii's problem.

The main result of the present paper provides an affirmative and sharp answer to Olevskii's problem.

\begin{theorem}\label{thm:main1}
    Let $H$ be a Hilbert space and let $\{ x_j \}_{j\in \N}\subset H$ be a Bessel sequence with Bessel bound $1$. Then there exist $\{ f_j \}_{j \in \N} \subseteq L^\infty([0,1])$ with $\| f_j \|_{L^\infty([0,1])} \leq 1$ for all $j$ and a constant $0 <K< \infty$ such that
    \begin{equation}
        \langle x_j,x_k \rangle = K \int_0^1 f_j(x) \overline{f_k(x)} \, dx, \quad j,k \in \N.
    \end{equation}
    Moreover, one can choose $K=1$ and this constant is optimal.
\end{theorem}

We remark that the optimality of $K=1$ follows directly from the observation that if $x_{j}$ is a unit vector then $1 = \langle x_{j},x_{j} \rangle = K \int_0^1 |f_{j}|^2 \, dx \leq K$.

\subsection{}
For the corresponding off-diagonal problem, Kashin proved in \cite{kashin2022observation} that for every finite Bessel sequence $x_1, \dots, x_N$ with Bessel bound $1$ in a real Hilbert space there exist functions $f_j\in L^\infty([0,1])$ with $\|f_j\|_{L^\infty([0,1])}\leq \sqrt 2$ such that
$$
\langle x_j,x_k\rangle
=
\int_0^1 f_j(t)f_k(t)\,dt,
\qquad j\neq k;
$$
in the normalization of \eqref{eq:GG}, this corresponds to the constant $K = 2$. Kashin notes explicitly that his theorem does not preserve the diagonal entries $j=k$. Theorem \ref{thm:main1} closes this gap: it preserves the diagonal entries, gives the sharp bound $K=1$, and holds for countably infinite sequences in arbitrary Hilbert spaces.

\subsection{}\label{Olevskii extension}
Notice that a system $\{ x_j \}_{j\in \N}$ is a Bessel sequence with Bessel bound $1$ if and only if the Gram matrix
\begin{equation}\label{eq:gram}
    G=(g_{jk})_{j,k\in\N}, \quad g_{jk}=\langle x_j,x_k\rangle
\end{equation}
satisfies $0 \leq G \leq I$ in the sense of positive semidefinite linear operators on $\ell^2(\N)$. Moreover, every infinite Hermitian matrix with $0 \leq G \leq I$ is represented in the form \eqref{eq:gram} for some $\{ x_j \}_{j\in \N}$.
Hence, Theorem \ref{thm:main1} implies that the entries $g_{jk}$ of every Hermitian matrix with $0 \leq G \leq I$ can be represented as inner products of the form
$$
g_{jk} = \int_0^1 f_j(x) \overline{f_k(x)} \, dx,
$$
where $\{ f_j \}_{j \in \N}$ belong to the unit ball of $L^\infty([0,1])$.
With this Gram matrix formulation in hand, one can establish the relation between Theorem \ref{thm:main1} and the original formulation of Olevskii's problem in terms of an extension problem. To see this relation, let $E \subseteq [0,1]$ be measurable and assume that
$$
\nu := \lambda([0,1]\setminus E)>0,
$$
where $\lambda$ denotes the Lebesgue measure. Given a Bessel sequence $\{ u_j \}_{j\in \N} \subseteq L^2(E)$ with Bessel bound $1$, we ask whether there exists an orthonormal system $\{ \phi_j \}_{j\in \N} \subseteq L^2([0,1])$ so that $\phi_j = u_j$ on $E$ and $\phi_j$ is bounded on $A : = [0,1] \setminus E$ by a constant independent of $j$. Hence, we are asking for a pointwise controlled extension to an orthonormal system. To see the connection with Theorem \ref{thm:main1}, define
$$
G_u=
\bigl(\langle u_j,u_k\rangle_{L^2(E)}\bigr)_{j,k\in \N}.
$$
The Bessel condition gives
$
0\le G_u\le I
$
and therefore the matrix
$
D=I-G_u
$
is again positive semidefinite with $0 \leq D \leq I$. We now apply
Theorem \ref{thm:main1} to a Bessel sequence with Gram matrix $D$. After rescaling, we obtain functions
$v_j\in L^\infty(A)$ satisfying
$$
\langle v_j,v_k\rangle_{L^2(A)}
=
\delta_{jk}
-
\langle u_j,u_k\rangle_{L^2(E)}, \quad \|v_j\|_{L^\infty(A)}
\leq \frac{1}{\sqrt{\nu}},
$$
where $\delta_{jk}$ denotes the Dirac delta. Consequently, the system $\{ \phi_j\}_{j \in \N}$ defined by
$$
\phi_j(x)
=
\begin{cases}
u_j(x), & x \in E,\\
v_j(x), & x \in A,
\end{cases}
$$
forms an orthonormal system with the desired property. We therefore obtain the following sharp solution to Olevskii's extension problem.

\begin{theorem}\label{thm:main2}
Let $E\subseteq[0,1]$ be measurable such that $[0,1]\setminus E$ has positive Lebesgue measure.
Further, let $\{ u_j \}_{j\in \N} \subseteq L^2(E)$
be a Bessel sequence with Bessel bound $1$. Then there exists an orthonormal system $\{ \phi_j \}_{j\in \N}\subseteq L^2([0,1])$ satisfying
$$
\phi_j|_E=u_j,
\qquad
\|\phi_j\|_{L^\infty([0,1]\setminus E)}
\le
\frac{1}{\sqrt{\lambda([0,1]\setminus E)}},
\qquad
j\in \N.
$$
Moreover, the constant $\frac{1}{\sqrt{\lambda([0,1]\setminus E)}}$ is optimal for every such set $E$.
\end{theorem}

Without the pointwise $L^\infty$ constraint, the Bessel condition is precisely the classical Schur criterion for the existence of an orthonormal extension: a system in $L^2(E)$ extends to an orthonormal system in $L^2([0,1])$ if and only if it is a Bessel sequence with Bessel bound $1$. See \cite[p.~70]{KS35} and \cite[p.~934]{kashinSzarek2003knaster} for details.
The study of uniformly bounded orthogonal systems goes back to Menchoff \cite{menchoff1938}, whose work already contains constructions of $\pm 1$-valued systems with prescribed off-diagonal Gram entries, while Olevskii studied the related problem of extending systems of functions to complete orthonormal systems \cite{olevskii1969extension}. In this context, we point out that in \cite[p.~58]{Olevskii} it is asked under what conditions a system can be extended to a complete uniformly bounded orthonormal system. The completeness aspect is not addressed in the present paper.

\subsection*{Usage of Large Language Models}

Theorem \ref{thm:main1:finite} (for the case of real Hilbert spaces) was developed with the assistance of GPT-5.4. Specifically, the result of Ball and Prodromou \cite{ballProdromou2009}, which was known to the authors, was provided to GPT-5.4. GPT-5.4 was then guided toward a construction of the functions $f_j$ given in \eqref{fct}. The resulting functions provided a version of the main result for the case of an $N$-dimensional real Hilbert space. This result was subsequently generalized by the authors to the case of countably many vectors in an infinite-dimensional complex Hilbert space. The Lean verification was carried out with the assistance of GPT-5.5 via Codex.

\section{Proof of Theorem \ref{thm:main1} and Theorem \ref{thm:main2}}\label{sec:proofs}
Let us first introduce some notation. Given  a $d$-dimensional real inner product space $(X, \langle\cdot, \cdot\rangle)$ and a linear operator $T:X\to X$, we denote the trace of $T$ as  
\begin{equation*}
   \tr(T)\coloneqq\sum_{j=1}^d \langle Te_j,e_j\rangle 
\end{equation*}
where $(e_i)_{1\leq i\leq d}$ is an arbitrary orthonormal basis of $X$. Let the space of all symmetric real $N\times N$ matrices be denoted as $\sym_N$ and be equipped with the Hilbert-Schimidt inner product
\begin{equation*}
    \langle A, B\rangle \coloneqq \tr(AB).
\end{equation*}
For any $A\in \sym_N$, the spectral theorem guarantees the existence of an orthonormal basis of eigenvectors $v_1,\ldots,v_N$ of $A$ with corresponding eigenvalues $\lambda_1,\ldots,\lambda_N$. We define the positive spectral subspace of $A$ by
\begin{equation*}
    E(A)\coloneqq\lspan\{v_r:\lambda_r>0\}
\end{equation*}
and denote by $P_{E(A)}:\mathbb R^N\to E(A)$ the orthogonal projection onto $E(A)$. Let
\begin{equation*}
    A_+\coloneqq\sum_{\lambda_r>0}\lambda_r v_rv_r^T,
    \qquad
    A_-\coloneqq-\sum_{\lambda_r<0}\lambda_r v_rv_r^T.    
\end{equation*}
Then,
\begin{equation*}
    A=A_+-A_-,
    \qquad
    A_+\ge 0,
    \qquad
    A_-\ge 0,
    \qquad
    A_-P_E=0,
    \qquad
    A_+P_E=A_+.
\end{equation*}
For complex Hilbert spaces, we use the convention that the inner products are linear in the first variable and conjugate-linear in the second variable.

A key result that we will use is the following theorem, which was proved in~\cite[Thm.~1.4]{ballProdromou2009} to obtain a sharp combinatorial version of Vaaler’s theorem.
\begin{theorem}[Ball--Prodromou]\label{thm:ball_prodromou}
    Let $E$ be a finite-dimensional real inner product space with $\dim E=d$. Let $u_1,\dots,u_N\in E$ satisfy
    \begin{equation*}
       \sum_{j=1}^N u_j\otimes u_j=I_E 
    \end{equation*}
    and $Q:E\to E$ be self-adjoint and positive semi-definite. Then there exists $w\in E$ such that
    \begin{equation*}
    \langle Qw,w\rangle\ge \tr(Q),
    \qquad
    |\langle w,u_j\rangle|\le 1,\quad j=1,\dots,N.
    \end{equation*}
\end{theorem}
We will also need the characterization of convex sets by support functions~\cite[Cor.~13.1.1]{rockfellar1970}.
\begin{lemma}[Support functions]\label{lem:support_functions}
Let $V$ be a finite-dimensional real inner product space and let $K,L\subseteq V$ be nonempty compact convex sets. For $a\in V$, define
\begin{equation*}
    h_K(a)=\max_{x\in K}\langle a,x\rangle,
    \qquad
    h_L(a)=\max_{y\in L}\langle a,y\rangle. 
\end{equation*}
If $h_K(a)\le h_L(a)$ for every $a\in V$, then $K\subseteq L$.
\end{lemma}
Now we are ready to prove the following proposition, which will lead to a finite version of Theorem~\ref{thm:main1}.
\begin{proposition}\label{prop:real_convex_hull}
The set
\begin{equation*}
    \mathcal G=\{G\in \operatorname{Sym}_N:0\le G\le I_N\}
\end{equation*}
satisfies
\begin{equation*}
    \mathcal G\subseteq \operatorname{conv}\{xx^T:x\in [-1,1]^N\}.
\end{equation*}
\end{proposition}
\begin{proof}
Let us define
\begin{equation*}
    \mathcal C=\operatorname{conv}\{xx^T:x\in [-1,1]^N\}.
\end{equation*}
We will show that $\mathcal G \subseteq \mathcal C$ by invoking Lemma~\ref{lem:support_functions}. 

First, we prove that $\mathcal G$ and $\mathcal C$ are compact. Note that the set $\mathcal G$ is clearly closed, and if $G\in\mathcal G$, then all eigenvalues of $G$ lie in $[0,1]$, so $\|G\|_{\mathrm{HS}}\le \sqrt N$. Hence, $\mathcal G$ is bounded and therefore compact.

The set $\{xx^T:x\in [-1,1]^N\}$ is compact since $[-1,1]^N$ is compact in $\mathbb R^N$ and the map $x\longmapsto xx^T$ is continuous. Since the ambient space $\operatorname{Sym}_N$ is finite-dimensional, Carathéodory's theorem implies that the convex hull of this compact set is compact. Thus, $\mathcal C$ is compact.

Let $\alpha$ be the support function of $\mathcal G$:
\begin{equation*}
    \alpha(A)\coloneqq\max_{G\in\mathcal G}\langle A,G\rangle,
    \qquad A\in\operatorname{Sym}_N.
\end{equation*}
We claim that
\begin{equation*}
\alpha(A)=\tr(A_+).
\end{equation*}
Indeed, for $G\in\mathcal G$, we have
\begin{equation*}
\langle A,G\rangle
=
\tr(A_+G)-\tr(A_-G)
=
\tr(A_+) - \tr(A_+(I-G)) - \tr(A_-G).
\end{equation*}
Since $A_-,A_+,G,I-G\geq 0$, $\tr(A_+(I-G))$ and $\tr(A_-G)$ are non-negative. Thus, $\langle A,G\rangle\le \tr(A_+)$ for every $G\in\mathcal G$. Since $P_{E(A)}$ is an orthogonal projection, $0\le P_{E(A)}\le I$ and $P_{E(A)}\in\mathcal G$. Taking $G=P_{E(A)}$ yields
\begin{equation*}
\tr(AP_E)
=
\tr(A_+P_E)-\tr(A_-P_E)
=
\tr(A_+).
\end{equation*}
Hence, $\alpha(A)=\tr(A_+)$.

Let $\beta$ be the support function of $\mathcal C$:
\begin{equation*}
\beta(A)\coloneqq\max_{M\in\mathcal C}\langle A,M\rangle,
\qquad A\in\operatorname{Sym}_N.
\end{equation*}
Since $\mathcal C$ is the convex hull of the matrices $xx^T, x\in [-1,1]^N$, we have
\begin{equation*}
\beta(A)=\max_{x\in [-1,1]^N}\tr(Axx^T)
=
\max_{x\in [-1,1]^N}x^TAx.
\end{equation*}
Noting that $\tr(A_+)=\tr_{E(A)}(A)$, it remains to show that 
\begin{equation*}
    \max_{x\in [-1,1]^N}x^TAx\geq \tr_{E(A)}(A).
\end{equation*}
For that, we will invoke Theorem~\ref{thm:ball_prodromou} on the subspace $E(A)$. Let $e_1,\dots,e_N$ denote the standard basis of $\mathbb R^N$ and define
\begin{equation*}
u_j\coloneqq P_{E(A)} e_j,
\qquad j=1,\dots,N.
\end{equation*}
We observe that $\sum_{j=1}^N u_j\otimes u_j=I_E$. Indeed, for every $x\in E$, it holds that
\begin{equation*}
\sum_{j=1}^N \langle x,u_j\rangle u_j
=
\sum_{j=1}^N \langle x,P_{E(A)}e_j\rangle P_Ee_j
=
P_{E(A)}\left(\sum_{j=1}^N \langle x,e_j\rangle e_j\right)
=
P_{E(A)} x
=
x.
\end{equation*}
Since $A$ restricted to $E(A)$ is self-adjoint and positive semi-definite, Theorem~\ref{thm:ball_prodromou} guarantees the existence of $w\in E(A)$ such that
\begin{equation*}
w^TAw=\langle Aw,w\rangle\ge \tr_{E(A)}(A),
\qquad
|\langle w,u_j\rangle|\le 1,\quad j=1,\dots,N.
\end{equation*}
However,
\begin{equation*}
\langle w,u_j\rangle
=
\langle w,P_Ee_j\rangle
=
\langle P_Ew,e_j\rangle
=
\langle w,e_j\rangle,
\end{equation*}
which, together with $|\langle w,u_j\rangle|\le 1, 1\le j\le N$, implies that $w\in [-1,1]^N$. This shows that $\alpha(A)\le \beta(A)$ for every $A\in\sym_N$ and concludes the proof.
\end{proof}
Let $\herm_N$ denote the space of all Hermitian $N\times N$ matrices and denote the closed unit disk on $\C$ as $\D$. The following corollary states the complex counterpart of Proposition~\ref{prop:real_convex_hull}.
\begin{corollary}\label{cor:complex_convec_hull}
    The set
    \begin{equation*}
        \mathcal G=\{G\in \herm_N:0\le G\le I_N\}
    \end{equation*}
    satisfies
    \begin{equation*}
        \mathcal G\subseteq \conv\{\zeta\zeta^*:\zeta\in \D^N\}.
    \end{equation*}
\end{corollary}
\begin{proof}
    Let $G=A+iB\in\mathcal G$, $A,B\in \R^{N\times N}$. Since $G$ is Hermitian, $A^T=A$ and $B^T=-B$. Let us define the real symmetric $2N\times 2N$ matrix
    \begin{equation*}
    \widetilde G=
    \begin{pmatrix}
    A & -B \\
    B & A
    \end{pmatrix}.
    \end{equation*}
    Observe that, for $x,y\in\mathbb R^N$,  $u=x+iy\in\mathbb C^N$, it holds that
    \begin{equation*}
    \binom{x}{y}^T
    \widetilde G
    \binom{x}{y}
    =
    u^*Gu.
    \end{equation*}
    Since $G\ge 0$, this proves $\widetilde G\ge 0$. Applying the same argument to $I_N-G$ gives $I_{2N}-\widetilde G\ge 0$. Thus, $0\le \widetilde G\le I_{2N}$. By Proposition~\ref{prop:real_convex_hull} applied in dimension $2N$, there exists an integer $m\ge 1$, weights $p_1,\ldots,p_m\ge 0$ with $\sum_{r=1}^m p_r=1$, and vectors $q^{(1)},\ldots,q^{(m)}\in [-1,1]^{2N}$ such that
    \begin{equation*}
    \widetilde G
    =
    \sum_{r=1}^m p_r q^{(r)}(q^{(r)})^T.
    \end{equation*}
    Write
    \begin{equation*}
    q^{(r)}
    =
    \binom{a^{(r)}}{b^{(r)}},
    \qquad
    a^{(r)},b^{(r)}\in [-1,1]^N.
    \end{equation*}
    We have the identities 
    \begin{alignat*}{2}
        A&=\sum_{r=1}^m p_r a^{(r)}(a^{(r)})^T,& \qquad
        -B&=\sum_{r=1}^m p_r a^{(r)}(b^{(r)})^T, \\
        B&=\sum_{r=1}^m p_r b^{(r)}(a^{(r)})^T,& \qquad 
        A&=\sum_{r=1}^m p_r b^{(r)}(b^{(r)})^T. 
    \end{alignat*}
    For each $r$, define
    \begin{equation*}
    \zeta^{(r)}
    \coloneqq
    \frac{a^{(r)}+i b^{(r)}}{\sqrt 2}
    \in\mathbb C^N.
    \end{equation*}
    Since $a^{(r)},b^{(r)}\in [-1,1]^N$, $\zeta^{(r)}\in \D^N$. For $1\le j,k\le N$, we have
    \begin{equation*}
        \sum_{r=1}^m p_r
        \zeta_j^{(r)}\overline{\zeta_k^{(r)}}
        =
        \frac12(A_{jk}+A_{jk})
        +
        \frac{i}{2}(B_{jk}+B_{jk})
        =
        A_{jk}+iB_{jk}
        =
        G_{jk}.
    \end{equation*}
    Hence, $G= \sum_{r=1}^m p_r\zeta^{(r)}(\zeta^{(r)})^*$.
\end{proof}
In view of the relation between the Bessel condition and the condition on the Gram matrix observed in subsection~\ref{Olevskii extension}, the following theorem is a finite version of Theorem~\ref{thm:main1}.
\begin{theorem}\label{thm:main1:finite}
    Let $G=(g_{jk})_{1\le j,k\le N} \in \herm_N$ satisfy $0\leq G\leq I_N$ in the standard order. Then there exist $\{ f_j \}_{1\leq j\leq N} \subseteq L^\infty([0,1])$ with $\| f_j \|_{L^\infty([0,1])} \leq 1$ for all $j$ and a constant $0 <K< \infty$ such that
    \begin{equation}
        \langle x_j,x_k \rangle = K \int_0^1 f_j(x) \overline{f_k(x)} \, dx, \quad j,k \in \N.
    \end{equation}
    Moreover, one can choose $K=1$ and this constant is optimal.
\end{theorem}
\begin{proof}
    We prove the result for $K=1$. The optimality follows from the case of $N=1$. By Proposition~\ref{cor:complex_convec_hull}, there exists an integer $m\ge 1$, weights $p_1,\dots,p_m\ge 0$ with $\sum_{r=1}^m p_r=1$ and vectors $\zeta^{(1)},\ldots,\zeta^{(m)}\in \D^{N}$ such that
    \begin{equation*}
    G
    =
    \sum_{r=1}^m p_r \zeta^{(r)}(\zeta^{(r)})^*.
    \end{equation*}
    Without loss of generality, we assume that $p_r>0$ for all $r$. Define intervals
    \begin{equation*}
    I_r\coloneqq
    \left[\sum_{s<r}p_s,\sum_{s\le r}p_s\right)
    \qquad
    (1\le r<m)
    \end{equation*}
    and
    \begin{equation*}
    I_m\coloneqq
    \left[\sum_{s<m}p_s,1\right].
    \end{equation*}
    Note that $(I_r)_{1\le r\le m}$ forms a measurable partition of $[0,1]$ and $|I_r|=p_r$. For each $j=1,\dots,N$, define a simple function $f_j$ via
    \begin{equation}\label{fct}
    f_j(t)\coloneqq\zeta_j^{(r)},
    \qquad
    \forall t\in I_r, 
    \qquad
    1\le r\le m.
    \end{equation}
    Then $\|f_j\|_{L^\infty(0,1)}\le 1$. For $1\le j,k\le N$, it is easy to check that
    \begin{equation*}
    \int_0^1 f_j(t)\overline{f_k(t)}\,dt
    =
    \sum_{r=1}^m p_r
    \zeta_j^{(r)}\overline{\zeta_k^{(r)}}
    =
    G_{jk}
    =
    \langle z_j,z_k\rangle_H.\qedhere
    \end{equation*}
\end{proof}
To extend Theorem~\ref{thm:main1:finite} to countable index sets, we need to employ a limiting argument. To do this, let us first recall the following result; see e.g.~\cite[Thm.~3.2]{billingsley1971}.
\begin{lemma}[Realization of a probability measure]
\label{lem:probability_realization}
Let $X$ be a nonempty compact metric space and let $\nu$ be a Borel
probability measure on $X$. Then there exists a Borel measurable map
$T:[0,1]\to X$ such that
\[
\lambda\bigl(T^{-1}(B)\bigr)=\nu(B)
\]
for every Borel set $B\subseteq X$.
\end{lemma}
The following theorem represents Theorem~\ref{thm:main1} in the Gram matrix formulation.
\begin{theorem}\label{thm:countable_gram_realization}
    Let $H=(h_{jk})_{j,k\in \N}$ be a family of complex numbers such that, for all $N\in \N$, $H_N\coloneqq(h_{jk})_{1\leq j,k\leq N}\in \herm_N$ and $0\le H_N\le I_N$ in the standard order. Then there exists $\{ f_j \}_{j\in N} \subseteq L^\infty([0,1])$ with $\| f_j \|_{L^\infty([0,1])} \leq 1$ for all $j$ and a constant $0 <K< \infty$ such that
    \begin{equation}
        h_{jk} = K \int_0^1 f_j(x) \overline{f_k(x)} \, dx, \quad j,k \in \N.
    \end{equation}
    Moreover, one can choose $K=1$ and this constant is optimal.
\end{theorem}
\begin{proof}
Let us define $Z\coloneqq\D^{\N}$, so that $Z$ is a compact metrizable space in the product topology. For example, a compatible metric can be defined as
\begin{equation*}
d((y_n)_{n\in\N},(z_n)_{n\in\N})
\coloneqq
\sum_{n=1}^\infty
2^{-n}\min\{1,|y_{n}-z_{n}|\}.
\end{equation*}
Let $\mathcal P(Z)$ denote the set of Borel probability measures
on $Z$, equipped with the topology of weak convergence. This space
is compact. Indeed, by the Riesz--Markov--Kakutani representation theorem, $\mathcal P(Z)$ may be identified with the set of positive
functionals $L\in C(Z)^*$ satisfying $L(\mathds{1})=1$. This set is weak-*
closed in the unit ball of $C(Z)^*$, and hence is compact by the
Banach--Alaoglu theorem.

For $j,k\in J$, define
\begin{equation*}
\mathcal C_{jk}
=
\left\{
\nu\in\mathcal P(Z):
\int_X z_j\overline{z_k}\,d\nu(x)=h_{jk}
\right\}.
\end{equation*}
Since $z\longmapsto z_j\overline{z_k}$ is continuous on $Z$, $\mathcal C_{jk}$ is closed in
$\mathcal P(Z)$. We now show that the family $\{\mathcal C_{jk}:j,k\in \N\}$ has the finite intersection property. Let $\mathcal S\subseteq \N\times \N$ be finite, and let $F\subseteq \N$ be the finite set consisting of all indices occurring in the pairs in $\mathcal S$. Let us define the Gram matrix of $\{ x_j \}_{j \in F}$ as 
\begin{equation*}
    H_F= (h_{jk})_{j,k\in F}\coloneqq (\langle x_j, x_k\rangle)_{j,k\in F}.
\end{equation*}
By our assumption on $H$, $0\le H_F\le I_F$. Applying Theorem~\ref{thm:main1:finite} gives functions $\{a_j^{(F)}\}_{j\in F}\subseteq L^\infty([0,1])$ such that
\begin{equation*}
\|a_j^{(F)}\|_{L^\infty([0,1])}\le1,
\qquad
\int_0^1
a_j^{(F)}(t)\overline{a_k^{(F)}(t)}\,dt
=
h_{jk},
\qquad j,k\in F.
\end{equation*}
After changing these functions on a common null set, we may assume
that
\begin{equation*}
\forall t\in [0,1],\, \forall j\in F:\quad\left|a_j^{(F)}(t)\right|\le1.
\end{equation*}
Let $\eta_F$ be the distribution on
$\D^{\,F}$ of the random vector
\begin{equation*}
    t\longmapsto
    \bigl(a_j^{(F)}(t)\bigr)_{j\in F},
\end{equation*}
and let $\iota_F:\D^{\,F}\to Z$ be the continuous map which leaves the coordinates in $F$
unchanged and sets every coordinate outside $F$ equal to zero. Define
\begin{equation*}
\nu_F\coloneqq(\iota_F)_\#\eta_F.
\end{equation*}
Then
\begin{equation*}
    \int_X z_j\overline{z_k}\,d\nu_F(z)
    =
    \int_{\D^F} z_j\overline{z_k}\, d\eta_F(z)
    = h_{jk},
    \qquad j,k\in F.
\end{equation*}
Consequently, $\nu_F\in
\bigcap_{(j,k)\in\mathcal S}\mathcal C_{jk}$. This proves the finite intersection property.

Since $\mathcal P(Z)$ is compact and $(\mathcal C_{jk})_{j,k\in\N}$ are closed, there exists $\nu\in\bigcap_{j,k\in J}\mathcal C_{jk}$. Lemma~\ref{lem:probability_realization} then gives a Borel map
\begin{equation*}
T:[0,1]\to Z
\end{equation*}
whose distribution is $\nu$. Let $\pi_j:Z\to\D$
be the $j$-th coordinate map and define
\begin{equation*}
f_j(t)=\pi_j(T(t)).
\end{equation*}
Then $|f_j(t)|\le1$ for every $t$ and
\begin{equation*}
\int_0^1 f_j(t)\overline{f_k(t)}\,dt
=
\int_Z z_j \overline{z_k}\, d\nu(z)
=
h_{jk},
\qquad j,k\in J.\qedhere
\end{equation*}
\end{proof}
\begin{lemma}\label{lem:rescaling}
    Let $\{f_j\}_{j\in\N}\subseteq L^\infty([0,1])$ and $A\subseteq [0,1]$ with $\lambda(A)>0$. Then there exist $\{v_j\}_{j\in\N}\subseteq L^\infty(A)$ such that 
    \begin{equation*}
        \|v_j\|_{L^\infty(A)} = \frac{1}{\sqrt{\lambda(A)}}\|f_j\|_{L^\infty([0,1])},
        \qquad 
        \langle v_j, v_k\rangle_{L^2(A)}
        =
        \langle f_j, f_k\rangle_{L^2([0,1])},
        \qquad
        j,k\in \N.
    \end{equation*}
\end{lemma}
\begin{proof}
    Let $\delta\coloneqq \lambda(A)$ and $d\mu_A(x)
    \coloneqq
    \frac{1}{\delta}\mathds 1_A(x)\,dx$
    be the normalized Lebesgue measure on $A$. Let us define
    \begin{equation*}
    \sigma_A(x)
    \coloneqq
    \frac{\lambda(A\cap[0,x])}{\delta},
    \qquad 0\le x\le1.
    \end{equation*}
    For $0\le x\le y\le1$, we have
    \begin{equation*}
    0\le \sigma_A(y)-\sigma_A(x)
    =
    \frac{\lambda(A\cap(x,y])}{\delta}
    \le \frac{y-x}{\delta},
    \end{equation*}
    which implies that $\sigma_A$ is continuous and non-decreasing. Moreover,
    \begin{equation*}
    \sigma_A(0)=0,
    \qquad
    \sigma_A(1)=1.
    \end{equation*}
    We claim that the distribution of $\sigma_A$ under $\mu_A$ is the Lebesgue measure on $[0,1]$. Fix $s\in[0,1]$. By continuity and the
    intermediate value theorem, the level set $\sigma_A^{-1}(\{s\})$ is nonempty and compact. Let us denote $b_s\coloneqq\max\sigma_A^{-1}(\{s\})$. Since $\sigma_A$ is non-decreasing, it holds that
    \begin{equation*}
    \sigma_A^{-1}([0,s])=[0,b_s].
    \end{equation*}
    Therefore
    \begin{equation*}
    \mu_A\bigl(\sigma_A^{-1}([0,s])\bigr)
    =
    \mu_A([0,b_s])
    =
    \frac{\lambda(A\cap[0,b_s])}{\delta}
    =
    \sigma_A(b_s)
    =
    s,
    \end{equation*}
    which proves the claim. For $x\in A$, define
    \begin{equation*}
    v_j(x)
    \coloneqq
    \frac{1}{\sqrt{\delta}}\,
    f_j(\sigma_A(x)).
    \end{equation*}
    It then holds that
    \begin{gather*}
        \|v_j\|_{L^\infty(A)} = \frac{1}{\sqrt{\delta}}\|f_j\|_{L^\infty([0,1])},
        \\
        \int_A v_j(x)\overline{v_k(x)}\,dx
        =
        \frac{1}{\delta}
        \int_A
        f_j(\sigma_A(x))
        \overline{f_k(\sigma_A(x))}\,dx
        =
        \int_0^1 f_j(t)\overline{f_k(t)}\,dt.\qedhere
    \end{gather*}
\end{proof}
\begin{proof}[Proof of Theorem~\ref{thm:main2}]
    Let us denote
    \begin{equation*}
        A\coloneqq[0,1]\setminus E,
        \qquad
        \delta\coloneqq\lambda(A)>0,
    \end{equation*}
    and define the Gram matrix of $\{u_j\}_{j\in\N}$ as 
    \begin{equation*}
    G
    \coloneqq
    \bigl(
    \langle u_j,u_k\rangle_{L^2(E)}
    \bigr)_{j,k\in \N}.
    \end{equation*}
    Since $\{u_j\}_{j\in\N}$ is a Bessel sequence, the restriction of $G$ to $1\leq j\leq N$, denoted as $G_N$, satisfies $0\leq G_N\leq I_N$ for every $N\in \N$. Hence, the restriction of the infinite Hermitian matrix $H\coloneqq I-G$ satisfies the same condition. By Theorem~\ref{thm:countable_gram_realization}, there exist functions $\{f_j\}_{j\in\N}\subseteq L^\infty([0,1])$ such that
    \begin{equation*}
        \|f_j\|_{L^\infty([0,1])}\le1,
        \qquad
        \langle f_j, f_k\rangle_{L^2([0,1])}
        =
        h_{jk},
        \qquad j,k\in \N.
    \end{equation*}
    Lemma~\ref{lem:rescaling} then yields $\{v_j\}_{j\in\N}\subseteq L^\infty(A)$ such that
    \begin{equation*}
        \|v_j\|_{L^\infty(A)}\le \frac{1}{\sqrt{\delta}},
        \qquad
        \langle v_j, v_k\rangle_{L^2(A)}
        =
        h_{jk},
        \qquad j,k\in \N.
    \end{equation*} 
    Define $\{\phi_j\}_{j\in\N}$ via
    \begin{equation*}
    \phi_j(x)
    \coloneqq
    \begin{cases}
    u_j(x),&x\in E,\\
    v_j(x),&x\in A.
    \end{cases}
    \end{equation*}
    Then $\phi_j|_E=u_j$, $\|\phi_j\|_{L^\infty(A)} \le \frac{1}{\sqrt{\delta}}$. For $j,k\in\N$, we check that
    \begin{align*}
    \langle\phi_j,\phi_k\rangle_{L^2([0,1])}
    &=
    \langle u_j,u_k\rangle_{L^2(E)}
    +
    \langle v_j,v_k\rangle_{L^2(A)} \\
    &=
    \langle u_j,u_k\rangle_{L^2(E)}
    +
    \delta_{jk}
    -
    \langle u_j,u_k\rangle_{L^2(E)}
    =
    \delta_{jk}.
    \end{align*}
    Thus, $\{\phi_j\}_{j\in \N}$ is an orthonormal system.
    
    Finally, we prove the optimality of the constant. Take
    \begin{equation*}
    J=\{1\},
    \qquad
    u_1=0\in L^2(E).
    \end{equation*}
    Let $\phi_1$ be an orthonormal extension of $u_1$. Then,
    \begin{equation*}
    1
    =
    \|\phi_1\|_{L^2([0,1])}^2
    =
    \int_A|\phi_1(x)|^2\,dx
    \le
    \delta\|\phi_1\|_{L^\infty(A)}^2,
    \end{equation*}
    which implies that $\|\phi_1\|_{L^\infty(A)} \ge \frac{1}{\sqrt{\delta}}$.
\end{proof}

\section{Appendix: Lean formalization}

We now discuss a formalization of Theorem \ref{thm:main1} with the optimal constant $K=1$ in Lean 4 \cite{Moura2021Lean4}, building on the mathematical library Mathlib \cite{mathlib2020}. There are now several recent papers in analysis with complete Lean verifications, including  \cite{armstrong2026formalization, pont2026cantor,hariharan2026milestone,ilin2026semi, miller2026formalization}. The objective in this section is to precisely formalize the following theorem.

\begin{theorem*}
    Let $H$ be a complex Hilbert space and let $\{ x_j \}_{j\in \N}\subset H$ be a Bessel sequence with Bessel bound $1$. Then there exist measurable functions $\{ f_j \}_{j \in \N}$ defined on $[0,1]$ with
    $$
    | f_j(x) | \leq 1, \quad x \in [0,1], \quad j \in \N,
    $$
    such that for all $j,k \in \N$ one has
    \begin{equation}
        \langle x_j,x_k \rangle =  \int_0^1 f_j(x) \overline{f_k(x)} \, dx.
    \end{equation}
\end{theorem*}

The formalization of the latter theorem uses Mathlib \cite{mathlib2020} for the standard notions from measure theory and functional analysis. The purpose of this section
is to make the correspondence between the Lean formalization and the mathematical statement precise and transparent without assuming familiarity with Lean.

\subsection{The underlying measure space}

We start by representing the unit interval as a
subset of $\mathbb{R}$. To do so, we make a definition in Lean which we call \lean{unitInterval}. In the source code, all declarations of this appendix are placed inside a dedicated namespace; for readability, the namespace prefix is omitted in what follows. The definition reads as follows.

\begin{leancode}
def unitInterval : Set ℝ :=
  Set.Icc (0 : ℝ) 1
\end{leancode}

Here \lean{Set.Icc} denotes a closed interval with endpoints $0$ and $1$. The annotation \lean{(0 : ℝ)} tells Lean that the endpoints are real numbers. We equip this interval with the Lebesgue measure. In Mathlib, \lean{volume} is Lebesgue measure on $\mathbb{R}$. To obtain the Lebesgue measure on the unit interval, we restrict it using \lean{volume.restrict unitInterval} and abbreviate the resulting measure by \lean{μ}. The corresponding Lean code reads as follows.

\begin{leancode}
abbrev μ : Measure ℝ :=
  volume.restrict unitInterval
\end{leancode}

\subsection{The Bessel condition}

Mathlib does not contain the definition of a Bessel sequence. Hence, we define it manually as follows.

\begin{leancode}
def BesselSequenceBoundOne
    {H : Type*} [NormedAddCommGroup H] [InnerProductSpace ℂ H]
    (x : ℕ → H) : Prop :=
    ∀ (F : Finset ℕ) (c : ℕ → ℂ),
      ‖∑ j ∈ F, c j • x j‖ ^ 2 ≤ ∑ j ∈ F, ‖c j‖ ^ 2
\end{leancode}

The term \lean{F : Finset ℕ} represents a finite subset $F\subset\mathbb{N}$ and the symbol \lean{•} denotes scalar multiplication. Consequently, \lean{BesselSequenceBoundOne x} states
that
$$
    \left\|\sum_{j\in F}c_jx_j\right\|^2
    \leq
    \sum_{j\in F}|c_j|^2
$$
for every finite $F\subset\mathbb{N}$ and every choice of coefficients $c_j \in \C$. The coefficients are represented in Lean as functions from $\N$ to $\C$ via \lean{c : ℕ → ℂ}. We also note that Lean indexes $\mathbb{N}$ by $0,1,2,\ldots$.

\subsection{The formal statement}

With these definitions in place, the assertion of
Theorem~\ref{thm:main1} with $K=1$ can be stated in Lean as follows.

\begin{leancode}
theorem main
    (H : Type*) [NormedAddCommGroup H] [InnerProductSpace ℂ H] [CompleteSpace H]
    (x : ℕ → H) (hx : BesselSequenceBoundOne x) :
    ∃ f : ℕ → ℝ → ℂ,
      (∀ j, AEStronglyMeasurable (f j) μ) ∧
      (∀ j t, t ∈ unitInterval → ‖f j t‖ ≤ 1) ∧
      ∀ j k, ⟪x k, x j⟫_ℂ = ∫ t in unitInterval, f j t * star (f k t)
\end{leancode}

The theorem starts with an abstract type \lean{H}. We have to give this type the structure of a Hilbert space. In Lean this can be done by imposing the three hypotheses \lean{NormedAddCommGroup H},
\lean{InnerProductSpace ℂ H}, and \lean{CompleteSpace H} which say that $H$ is a complex Hilbert space. Moreover, the theorem takes as an input a sequence $\{x_j \}$ of elements in the Hilbert space. The hypothesis \lean{(hx : BesselSequenceBoundOne x)} uses the previously defined notion of a Bessel sequence and asserts that $\{x_j \}$ has Bessel bound $1$.

The conclusion of the theorem says that there is a family of functions $f_j:\mathbb{R}\to\mathbb{C}$ having three properties. The first one is \lean{(∀ j, AEStronglyMeasurable (f j) μ)} which means that every
$f_j$ is almost everywhere strongly measurable with respect to
$\mu$, i.e., $f_j$ agrees $\mu$-almost everywhere with a measurable function. The second one is \lean{(∀ j t, t ∈ unitInterval → ‖f j t‖ ≤ 1)} and gives the pointwise estimate
$
    |f_j(t)|\leq 1
$
for $t\in[0,1]$. The third one is given by
\begin{center}
    \lean{∀ j k, ⟪x k, x j⟫_ℂ = ∫ t in unitInterval, f j t * star (f k t)}
\end{center}
and this corresponds to
\begin{equation}\label{eq:conclusion}
    \langle x_j,x_k\rangle
    =
    \int_0^1 f_j(t)\overline{f_k(t)}\,dt,
    \qquad j,k\in\mathbb{N}.
\end{equation}
Mathlib takes complex inner products to be conjugate-linear in the first argument and linear in the second. With the convention used in this paper, which is linear in the first argument, the expression \lean{⟪x k, x j⟫_ℂ} therefore represents $\langle x_j,x_k\rangle$. Moreover, \lean{star (f k t)} is the complex conjugate of $f_k(t)$. Thus, the final property is exactly \eqref{eq:conclusion}.
Accordingly, \lean{main} formalizes the existence assertion for
countable Bessel sequences with Bessel bound $1$ and with
$K=1$. The separate assertion that $1$ is the optimal universal
constant is not part of the formalization, since it amounts to the one-line argument given after Theorem \ref{thm:main1}.

\subsection{Source code}

The complete Lean formalization is available at
\begin{center}   \url{https://github.com/lukasliehr/GrothendieckBessel}.
\end{center}
The file \lean{Showcase.lean} contains all definitions and theorem
statements presented above, with two occurrences of \lean{sorry}
replacing the corresponding proofs. The companion file
\lean{Showcase_WithProofs.lean} contains the same declarations, with
both placeholders replaced by complete, machine-checked proofs.

\section*{Acknowledgments}

L.L.~is grateful to the Azrieli Foundation for the award of an Azrieli Fellowship and acknowledges the support of this research by ISF Grant No.~854/25.

\bibliographystyle{plain}
\bibliography{bibfile}

\begin{thebibliography}{10}

\bibitem{alonEtAl2006quadratic}
Noga Alon, Konstantin Makarychev, Yury Makarychev, and Assaf Naor.
\newblock Quadratic forms on graphs.
\newblock {\em Invent. Math.}, 163(3):499--522, 2006.

\bibitem{armstrong2026formalization}
Scott Armstrong and Julia Kempe.
\newblock Formalization of {D}e {G}iorgi--{N}ash--{M}oser {T}heory in {L}ean.
\newblock {\em arXiv preprint arXiv:2604.05984}, 2026.

\bibitem{ballProdromou2009}
Keith~M. Ball and Maria Prodromou.
\newblock A sharp combinatorial version of {Vaaler's} theorem.
\newblock {\em Bull. Lond. Math. Soc.}, 41(5):853--858, 2009.

\bibitem{billingsley1971}
Patrick Billingsley.
\newblock {\em Weak convergence of measures: {A}pplications in probability}, volume No. 5 of {\em Conference Board of the Mathematical Sciences Regional Conference Series in Applied Mathematics}.
\newblock Society for Industrial and Applied Mathematics, Philadelphia, PA, 1971.

\bibitem{charikarWirth2004}
Moses Charikar and Anthony Wirth.
\newblock Maximizing quadratic programs: extending {Grothendieck's} inequality.
\newblock In {\em 45th Annual IEEE Symposium on Foundations of Computer Science (FOCS 2004)}, pages 54--60. IEEE Computer Society, 2004.

\bibitem{pont2026cantor}
Jaume de~Dios de~Dios~Pont, Lukas Liehr, and Mitchell~A Taylor.
\newblock Cantor measures with odd base do not admit {F}ourier frames.
\newblock {\em arXiv preprint arXiv:2607.08656}, 2026.

\bibitem{Moura2021Lean4}
Leonardo de~Moura and Sebastian Ullrich.
\newblock The {Lean} 4 theorem prover and programming language.
\newblock In {\em Automated Deduction -- CADE 28}, volume 12699 of {\em Lecture Notes in Computer Science}, pages 625--635, Cham, 2021. Springer.

\bibitem{grothendieck1953resume}
Alexander Grothendieck.
\newblock R\'esum\'e de la th\'eorie m\'etrique des produits tensoriels topologiques.
\newblock {\em Bol. Soc. Mat. S\~ao Paulo}, 8:1--79, 1953.

\bibitem{hariharan2026milestone}
Sidharth Hariharan, Christopher Birkbeck, Seewoo Lee, Ho~Kiu~Gareth Ma, Bhavik Mehta, Auguste Poiroux, and Maryna Viazovska.
\newblock A {M}ilestone in {F}ormalization: {T}he {S}phere {P}acking {P}roblem in {D}imension 8.
\newblock {\em arXiv preprint arXiv:2604.23468}, 2026.

\bibitem{ilin2026semi}
Vasily Ilin.
\newblock Semi-autonomous formalization of the {V}lasov-{M}axwell-{L}andau equilibrium.
\newblock {\em arXiv preprint arXiv:2603.15929}, 2026.

\bibitem{KS35}
Stefan Kaczmarz and Hugo Steinhaus.
\newblock {\em Theorie der Orthogonalreihen}, volume~VI of {\em Monografje Matematyczne}.
\newblock Seminarjum Matematyczne Uniwersytetu Warszawskiego, Warszawa--Lw\'ow, 1936.

\bibitem{kashin2022observation}
B.~S. Kashin.
\newblock An observation on the {Gram} matrices of systems of uniformly bounded functions and a problem of {Olevskii}.
\newblock {\em Russian Math. Surveys}, 77(1):171--173, 2022.
\newblock Translated from Uspekhi Mat. Nauk 77 (2022), no.~1, 183--184.

\bibitem{kashinSzarek2003gram}
B.~S. Kashin and S.~J. Szarek.
\newblock On the {Gram} matrices of systems of uniformly bounded functions.
\newblock {\em Proc. Steklov Inst. Math.}, 243:227--233, 2003.
\newblock Translated from Tr. Mat. Inst. Steklova 243 (2003), 237--243.

\bibitem{kashinSzarek2003knaster}
Boris~S. Kashin and Stanislaw~J. Szarek.
\newblock The {Knaster} problem and the geometry of high-dimensional cubes.
\newblock {\em C. R. Math. Acad. Sci. Paris}, 336(11):931--936, 2003.

\bibitem{megretski2001}
Alexandre Megretski.
\newblock Relaxations of quadratic programs in operator theory and system analysis.
\newblock In {\em Systems, Approximation, Singular Integral Operators, and Related Topics (Bordeaux, 2000)}, volume 129 of {\em Oper. Theory Adv. Appl.}, pages 365--392. Birkh\"auser, Basel, 2001.

\bibitem{menchoff1938}
D.~Menchoff.
\newblock Sur les s\'eries de fonctions orthogonales born\'ees dans leur ensemble.
\newblock {\em Rec. Math. [Mat. Sbornik] N.S.}, 3(45)(1):103--120, 1938.

\bibitem{miller2026formalization}
Joseph~K. Miller.
\newblock A {F}ormalization of the {M}ean-{F}ield {D}erivation of the {V}lasov {E}quation: {AI-Assisted Lean Formalization as a Strategy Game}.
\newblock {\em arXiv preprint arXiv:2607.08986}, 2026.

\bibitem{nemirovskiRoosTerlaky1999}
Arkadi Nemirovski, Cornelis Roos, and Tam\'as Terlaky.
\newblock On maximization of quadratic form over intersection of ellipsoids with common center.
\newblock {\em Math. Program.}, 86(3):463--473, 1999.

\bibitem{olevskii1969extension}
A.~M. Olevskii.
\newblock On the extension of a sequence of functions to a complete orthonormal system.
\newblock {\em Math. Notes Acad. Sci. USSR}, 6(6):908--913, 1969.
\newblock Translated from Mat. Zametki 6 (1969), no.~6, 737--747.

\bibitem{Olevskii}
A.~M. Olevskii.
\newblock {\em Fourier Series with Respect to General Orthogonal Systems}, volume~86 of {\em Ergebnisse der Mathematik und ihrer Grenzgebiete}.
\newblock Springer-Verlag, New York--Heidelberg, 1975.

\bibitem{pisier2012grothendieck}
Gilles Pisier.
\newblock Grothendieck's theorem, past and present.
\newblock {\em Bull. Amer. Math. Soc. (N.S.)}, 49(2):237--323, 2012.

\bibitem{rockfellar1970}
R.~Tyrrell Rockafellar.
\newblock {\em Convex analysis}, volume No. 28 of {\em Princeton Mathematical Series}.
\newblock Princeton University Press, Princeton, NJ, 1970.

\bibitem{mathlib2020}
{The mathlib Community}.
\newblock The {Lean} mathematical library.
\newblock In {\em Proceedings of the 9th ACM SIGPLAN International Conference on Certified Programs and Proofs (CPP 2020)}, pages 367--381, New York, NY, 2020. ACM.

\end{thebibliography}

\end{document}